\documentclass[12pt]{amsart}
\usepackage[utf8]{inputenc}

\usepackage{multirow}

\usepackage{my_preamble}
\usepackage{comment}

\usepackage{amssymb}
\usepackage{mathtools}

\usepackage{crossreftools}

\makeatletter
\newtheorem*{rep@theorem}{\rep@title}
\newcommand{\newreptheorem}[2]{%
\newenvironment{rep#1}[1]{%
 \def\rep@title{#2 \ref{##1}}%
 \begin{rep@theorem}}%
 {\end{rep@theorem}}}
\makeatother

\newtheorem*{theorem*}{Theorem}
\newreptheorem{theorem}{Theorem}

\makeatletter
\newtheorem*{rep@definition}{\rep@title}
\newcommand{\newrepdefinition}[2]{%
\newenvironment{rep#1}[1]{%
 \def\rep@title{#2 \ref{##1}}%
 \begin{rep@definition}}%
 {\end{rep@definition}}}
\makeatother

\newreptheorem{definition}{Definition}

\makeatletter
\newtheorem*{rep@proposition}{\rep@title}
\newcommand{\newrepproposition}[2]{%
\newenvironment{rep#1}[1]{%
 \def\rep@title{#2 \ref{##1}}%
 \begin{rep@proposition}}%
 {\end{rep@proposition}}}
\makeatother

\newreptheorem{proposition}{Proposition}

\makeatletter
\newtheorem*{rep@question}{\rep@title}
\newcommand{\newrepquestion}[2]{%
\newenvironment{rep#1}[1]{%
 \def\rep@title{#2 \ref{##1}}%
 \begin{rep@question}}%
 {\end{rep@question}}}
\makeatother

\newtheorem{question}[thm]{Question}
\newreptheorem{question}{Question}

\DeclareRobustCommand{\qedify}[1]{%
  \ifmmode \quad\hbox{#1}
  \else
    \leavevmode\unskip\penalty9999 \hbox{}\nobreak\hfill
    \quad\hbox{#1}%
  \fi
}
\newenvironment{example}{\begin{example*}\pushQED{\qedify{$\diamondsuit$}}}{\popQED\end{example*}}

\DeclarePairedDelimiter\floor{\lfloor}{\rfloor}

\DeclareMathOperator*{\rank}{r_M}
\DeclareMathOperator*{\nulll}{n_M}
\DeclareMathOperator*{\Aut}{Aut}

\DeclareMathOperator*{\ranki}{r_{M_i}}

\DeclareMathOperator*{\rankiii}{r_{M_{i+1}}}

\DeclareMathOperator*{\rankMn}{r_{M_n}}
\DeclareMathOperator*{\rankMr}{r_{M_r}}
\DeclareMathOperator*{\rankMz}{r_{M_0}}

\DeclareMathOperator*{\ch}{ch}

\newcommand{\MW}{\mathrm{MW}}

\title{The rank-nullity ring of a matroid}

\author{Tara Fife}
\thanks{\href{mailto:fi.tara@gmail.com}{\texttt{fi.tara@gmail.com}}}

\author{Eline Mannino}
\thanks{\href{mailto:elinmann@math.uio.no}{\texttt{elinmann@math.uio.no}}}

\author{Felipe Rinc\'on}
\thanks{\href{mailto:f.rincon@qmul.ac.uk}{\texttt{f.rincon@qmul.ac.uk}}}

\date{}

\usepackage[bb=boondox,bbscaled=.95,cal=boondoxo]{mathalpha}

\begin{document}

\begin{abstract}
    We introduce the rank-nullity ring of a matroid $M$, which is a subring of the Chow ring of the permutahedral toric variety. This subring contains the tautological Chern classes of $M$, a fact we deduce from a highly symmetric formula for these classes. When the matroid $M$ is a uniform matroid, the rank-nullity ring coincides with the subring of $S_n$-invariants of the Chow ring of the permutahedral toric variety. In this case, we compute its Hilbert function explicitly and provide a Gr\"obner basis for the ideal of relations among its generators. 
\end{abstract}

\maketitle

\section{Introduction}
A matroid is a combinatorial abstraction of the concept of independence that arises in various areas of mathematics, including linear algebra and graph theory. Several geometric models of matroids have been explored in the literature, among them the matroid polytope, the Bergman fan, and the conormal fan; see \cite{ardila2022geometry} for a survey. 

The framework of \textit{tautological classes of matroids} introduced in \cite{berget2023tautological} provides a unified perspective on these seemingly distinct models, allowing the authors to recover and generalize numerous known results, prove conjectures about their connections, and establish several new findings.

More specifically, given a matroid $M$, the authors of \cite{berget2023tautological} define the \textit{tautological sub $K$-class} $[S_M]$ and the \textit{tautological quotient $K$-class} $[Q_M]$ as certain torus-equivariant classes on the permutohedral toric variety, generalizing the classes of the universal subbundle and quotient bundle associated with a linear subspace (see \cite[Definition 3.9]{berget2023tautological}).  
Their Chern roots $c_i(S_M), c_i(Q_M)$, called the {\em tautological Chern classes} of $M$, lie in the Chow ring $A^*(U_{n,n})$ of the permutohedral toric variety, which has a standard presentation by generators and relations:
$$A^*(U_{n,n})=\frac{\mathbb{Z}[x_S : \emptyset \subsetneq S \subseteq [n]]}{\langle x_S x_T : S \nsubseteq T \text{ and } S \nsupseteq T\rangle+\langle\sum_{S\ni e} x_S : e\in [n] \rangle }.$$ 
In \cite[Remark III.1]{berget2023tautological}, the authors provide a combinatorial formula for the generating functions of $c_i(S_M)$ and $c_i(Q_M)$, which we adapt in this work to establish the following explicit description. 

\begin{reptheorem}{mainformula} Let $M$ be a matroid of rank $r$ on the set $[n] := \{1,\dots,n\}$, with rank function $\rank:2^{[n]}\to \ZZ_{\geq 0}$. For any $0\leq k \leq r$, the Chern class $c_k(S_M)$ of the tautological subbundle $S_M$ is given by 
$$c_k(S_M)=\sum \alpha_S(S_1, \dots, S_t, p_1, \dots, p_t) \, x_{S_1}^{p_1}x_{S_2}^{p_2}\cdots x_{S_t}^{p_t} \quad \in A^k(U_{n,n}),
$$
where the sum is taken over all chains of subsets $\emptyset \subsetneq S_1 \subsetneq \cdots \subsetneq S_t \subseteq E$ and all positive integers $p_1, \dots,p_t$ satisfying $\sum_{i=1}^t p_i = k$, with coefficients given by
$$\alpha_S(S_1, \dots, S_t, p_1, \dots, p_t) :=
(-1)^k\binom{\rank(S_1)}{p_1}
\binom{\rank(S_2)-\Tilde{p}_1}{p_2}\dots
\binom{\rank(S_t)-\Tilde{p}_{t-1}}{p_t}$$
where $\Tilde{p}_i=\sum_{j=1}^ip_j$.

A similar formula holds for the Chern classes $c_k(Q_M)$ of the tautological quotient bundle $Q_M$; see \cref{mainformula} in \cref{subsec:formula}. 
\end{reptheorem}

In \cref{corcsm}, we specialize the formula above to obtain an explicit description of the Poincaré duals $\ch_k(M)$ of the Chern-Schwartz-MacPherson cycles of the matroid $M$ (see \cite{de2020chern}) in the Chow ring $A^*(M)$ of $M$. The second author has also used an earlier version of \cref{corcsm} to prove results regarding the distribution and emergence of \textit{Chern numbers of matroids} of rank $3$, which are the numbers obtained by intersecting various CSM-cycles of a matroid (see \cite{mannino2023chern}).

Motivated by a preliminary version of this paper in which our formula for the CSM classes $\ch_k(M)$ was still conjectural, an independent proof of it was also given in \cite[Theorem 3.1 and Theorem 4.2]{caraballo2024staircase}. In that work, the authors also provide a geometric argument in the case the matroid $M$ is realizable over $\CC$. 

The symmetry of the formula in \cref{mainformula}, in which the coefficients $\alpha_S$ depend only on the sizes of the subsets $S_j$, implies that the tautological Chern classes $c_i(S_M), c_i(Q_M)$ of a matroid live in a particular subring $R^*(M)$ of $A^*(U_{n,n})$, which we call the \textit{rank-nullity ring} of the matroid $M$.
\begin{repdefinition}{def:ranknullity}
Let $M$ be a matroid on the set $[n]$ with rank function $\rank$ and nullity function $\nulll$ (given by $\nulll(S) := |S| - \rank(S)$). 
For $0\leq i,j \leq n$, let
$$y_{i,j} := \sum_{\substack{\emptyset \neq S \subseteq [n] \\ \rank(S)=i,\ \nulll(S)=j}}x_S \quad \in A^*(U_{n,n}).$$

The \emph{rank-nullity ring} of $M$ is the subring
$$R^*(M) := \mathbb{Z}[y_{i,j} \,:\, 0 \leq i, j \leq n] \quad \subseteq A^*(U_{n,n}).$$
\end{repdefinition}
Our main motivation for studying the rank-nullity ring is that, as we show in \cref{thmy}, it contains all the tautological Chern classes: for any $k$ we have
$$c_k(S_M), c_k(Q_M) \in R^*(M).$$
The rank-nullity ring $R^*(M)$ is in general much smaller than the Chow ring $A^*(U_{n,n})$; we quantify in \cref{prop:rel} exactly how many linearly independent generators this subring has. 

In \cref{propaut} we note that the rank-nullity ring $R^*(M)$ is invariant under any automorphism of the matroid $M$; in other words, we have 
$$R^*(M) \subseteq A^*(U_{n,n})^{\Aut(M)},$$
where $\Aut(M)$ denotes the automorphism group of $M$. 
The action of $\Aut(M)$ on the Chow ring $A^*(M)$ of $M$ was recently studied in \cite{angarone2025chow}, where the authors lift the Poincaré duality and Hard Lefschetz properties to the equivariant setting.

The rank-nullity ring of the uniform matroid $U_{r,n}$ of rank $r$ on $n$ elements is of particular interest. In this case, we show that it agrees with the subring of $S_n$-invariants of the permutohedral Chow ring $A^*(U_{n,n})$. We also compute a $\ZZ$-basis and the Hilbert function of this subring.
\begin{theorem*}[\cref{prop: RNU}, \cref{prop:basisRNU}, and \cref{prop:HfunctionU}]
For any positive integers $r \leq n$, the rank-nullity ring of the uniform matroids satisfies
    $$R^*(U_{r,n})= A^*(U_{n,n})^{S_n}.$$
It is a free $\ZZ$-module with basis 
$$\mathcal{B}_n=\{z_{s_1}^{p_1}z_{s_2}^{p_2}\cdots z_{s_l}^{p_l} : 1\leq s_1 < s_2 < \dots < s_l \leq n \text{ and } 1\leq p_i < s_i-s_{i-1}\},$$
where $z_j := \sum_{|S| = j}x_S$ and by convention $s_0 = 0$.
Its Hilbert function is given by 
\begin{align*}
{\rm rank}_\ZZ(R^d(U_{r,n})) = \binom{n-1}{d}.
\end{align*}
\end{theorem*}

We also investigate the ideal of relations among the generators $z_j$ of the rank-nullity ring $R^*(U_{r,n})= A^*(U_{n,n})^{S_n}$, and provide a Gr\"obner basis for it. 
\begin{reptheorem}{thm:grobnerbasis}
 The rank-nullity ring $R^*(U_{r,n})$ of the uniform matroids is the quotient of the free polynomial ring $\ZZ[Z_1, \dots, Z_n]$ by the ideal $\mathcal I$ generated by the polynomials
    \begin{equation*}
        Z_a \cdot \left( \sum_{i = b}^n \binom{i-a-1}{b-a-1} (Z_i + Z_{i+1} + \dots + Z_n )^{b-a} \right) \qquad \text{for $0\leq a < b \leq n$}
    \end{equation*}
    (with the convention that $Z_0 = 1$), where the class $[Z_i]$ in this quotient is equal to the element $z_i \in R^*(U_{r,n})$. 
    In fact, the polynomials above are a Gr\"obner basis for $\mathcal I$.
\end{reptheorem}

Finally, in the last section of the paper we further explore the behavior of the Hilbert functions of general rank-nullity rings (see \cref{tableH} for examples). 
In \cref{prop:Hincresing} we prove that for any matroid $M$ on the ground set $[n]$, the Hilbert function of $R^*(M)$ satisfies
    $${\rm rank}_\ZZ ({R^0(M)}) \leq {\rm rank}_\ZZ ({R^1(M)}) \leq \dots \leq {\rm rank}_\ZZ ({R^{\floor{\frac{n}{2}}}(M)}).$$
We do not know whether these Hilbert functions always satisfy stronger properties. 
\begin{repquestion}{qn:1}
Is the Hilbert function of the rank-nullity ring of a matroid $M$ always unimodal? Is it log-concave?
\end{repquestion}

\subsection*{Acknowledgements}
The authors thank Franquiz Caraballo Alba, Johannes Rau, and Kris Shaw for helpful conversations and suggestions. The authors also thank Philippe Nadeau for pointing out related work of Klyachko (see \cref{rem:klyachko}).

\section{Background and notation}

In this section we briefly recall all the background necessary for the paper. Readers already familiar with these topics might want to just skim through it or jump directly to Section \ref{sec:rank-nullity}.

\subsection{The Chow ring of a matroid}
Feichtner and Yuzvinsky \cite{feichtner2004chow} defined a graded commutative ring associated to an arbitrary matroid, as we recall below. The definition is motivated by the fact that, when a matroid arises from a complex hyperplane arrangement, the ring is isomorphic to the Chow ring of de Concini and Procesi \textit{wonderful compactification} of the complement of the arrangement \cite{de1995wonderful}.

\begin{defn}
\label[defn]{def:chow}
    Let $M$ be a matroid on the ground set $E=[n]=\{1,\dots,n\}$. The \textit{Chow ring} of $M$ is the quotient 
    \begin{equation}\label{eq:ChowringM}
    A^*(M)=\frac{\mathbb{Z}[x_F : F \text{ is a nonempty flat of }M]}{\langle x_F x_{F'} : F \nsubseteq F' \text{ and } F \nsupseteq F'\rangle+\langle\sum_{F\ni e} x_F : e\in E \rangle }.
\end{equation}
\end{defn}
The Chow ring of $M$ is a graded Artinian ring of socle degree ${\rm rank}(M)-1$, and its Hilbert series is a polynomial of that degree. 
Adiprasito, Huh, and Katz \cite{adiprasito2018hodge} proved that the Chow ring of any matroid satisfies several nice properties that we now recall. 

{\em Poincaré duality} states that, if the matroid $M$ has rank $r$, there is a group isomorphism 
$$A^{r-1}(M) \cong \ZZ$$
and moreover, for any $0 \leq k\leq r-1$, the multiplication map 
$$A^k(M)\times A^{r-1-k}(M) \longrightarrow A^{r-1}(M) \cong \ZZ$$
induces an isomorphism
$$A^{r-1-k}(M)\cong \mathrm{Hom}_\mathbb{Z}(A^k(M), \mathbb{Z}).$$
As a consequence, the \textit{Hilbert function} of the ring $A^*(M)$, which returns the dimension of its various graded components, is symmetric.

The {\em Hard Lefschetz} property states that there is an element $l\in A^1(M)$ such that, for every nonnegative integer $q \leq (r-1)/2$, multiplication by $l$ defines an isomorphism of vector spaces
\begin{align*}
L^q_l : A^q(M)_{\mathbb{R}} &\xlongrightarrow{\simeq} A^{r-1-q}(M)_{\mathbb{R}}\\
a \quad &\longmapsto \quad l^{r-1-2q} \cdot a.
\end{align*} 
This implies the weaker Lefschetz property: multiplication by the element $l \in A^1(M)$ from $A^q(M)_{\mathbb{R}}$ to $A^{q+1}(M)_{\mathbb{R}}$ is injective when $q \leq (r-1)/2$ and surjective when $q \geq (r-1)/2$.
It follows that the Hilbert function of the Chow ring $A^*(M)$ is unimodal, with peak around $q = (r-1)/2$. In fact, it is conjectured that the Hilbert function of $A^*(M)$ is log-concave, and further, that it is the sequence of coefficients of a real-rooted polynomial \cite{stevens2021real, ferroni2024valuative}.

\subsection{The Bergman fan and Minkowski weights}
Ardila and Klivans \cite{ardila2006bergman} assigned to any matroid $M$ a rational balanced polyhedral fan $\Sigma_M$, called the Bergman fan of $M$, having a cone for each chain of flats. 

We denote by $e_i$, with $i\in E$, the standard basis vectors of $\mathbb{Z}^E$, and by $N$ the quotient lattice $\mathbb{Z}^E/\mathbb{Z}\mathbf{1}$, where the vector $\mathbf{1}=\sum_{i\in E} e_i\in \mathbb{Z}^E$.
We also write $u_i$ for the image of $e_i$ in $N$, and $u_S = \sum_{i\in S} u_i$ for any subset $S\subseteq E$.

\begin{defn}
\label[defn]{def: Bergman}
Let $M$ be a loopless matroid of rank $r$ on the set $E$. The {\em Bergman fan} $\Sigma_M$ of $M$ is the pure $(r-1)$-dimensional polyhedral fan in $N_\mathbb{R}:=N \otimes \mathbb{R}$ consisting of the cones 
$$\sigma_{\mathcal{F}}:=\mathrm{cone}(u_{F_1},u_{F_2},\dots, u_{F_k}) \quad \subseteq N_{\mathbb{R}}$$
for each chain of flats $\mathcal{F}:\emptyset\subsetneq F_1\subsetneq \cdots \subsetneq F_k \subsetneq E$ in $M$.
\end{defn}
The Bergman fan $\Sigma_{U_{n,n}}$ of the free matroid $U_{n,n}$ on $n$ elements is a complete fan in $N_\RR$, called the {\em permutahedral fan}.

The group $\MW_{\ast}(\Sigma_M)$ of {\em Minkowski weights} supported on the Bergman fan $\Sigma_M$ consists of weight functions $\omega$ from the set of cones of $\Sigma_M$ to $\ZZ$ that satisfy the balancing condition  
--- see, for instance, \cite[Section 5.1]{adiprasito2018hodge} for a precise definition. There is a duality between Minkowski weights and elements of the Chow ring, as we now explain following the exposition in \cite[Section 5.2]{adiprasito2018hodge}. For a chain of flats $\emptyset\subsetneq F_1\subsetneq \cdots \subsetneq F_k \subsetneq E$ in $M$ corresponding to the cone $\sigma \in \Sigma_M$, let $x_\sigma := x_{F_1}x_{F_2}\cdots x_{F_k}\in A^k(M)$. The elements $x_\sigma$, with $\sigma$ a $k$-dimensional cone of $\Sigma_M$, generate the group $A^k(M)$.
The map 
\begin{align}
     t_{\Sigma_M}: \MW_k(\Sigma_M) & \xlongrightarrow{\simeq} \mathrm{Hom}(A^k(M), \mathbb{Z}) \label{eq:linear}\\
    \omega &\longmapsto t_{\Sigma_M}\omega, \qquad \text{where } t_{\Sigma_M}\omega\, (x_\sigma) := \omega(\sigma) \nonumber
\end{align}
is an isomorphism of groups, which is an analogue of the Kronecker duality map in algebraic topology. 
This isomorphism gives rise to the {\em cap product}: 
\begin{align}
    A^{l}(M)\times \MW_k(\Sigma_M) &\longrightarrow \MW_{k-l}(\Sigma_M) \label{eq:cap}\\
    (\xi, \omega) &\longmapsto \xi\cap \omega, \qquad \text{where } \xi\cap \omega \,(\sigma):=t_{\Sigma_M}\omega\,(\xi\cdot  x_\sigma). \nonumber
\end{align}
Let $\omega_{M}$ be the {\em Bergman class} of $M$, i.e., the top-dimensional Minkowski weight in $\MW_{r-1}(\Sigma_M)$ where every maximal cone has weight $1$.
As noted in \cite[Corollary~4.2.4]{backman2023simplicial}, we have an isomorphism
\begin{align}
    \delta_{\Sigma_{M}}: A^*(M)&\xlongrightarrow{\simeq} \MW_{r-1-*}(\Sigma_{M}) \label{eq:isomorphism}\\ 
    \xi &\longmapsto \xi\cap \omega_{M}, \nonumber
\end{align}
which is equal to the composition of the isomorphism in \cref{eq:linear} with Poincaré duality.

Finally, let $\Sigma_{U_{n,n}}$ be the permutahedral fan, i.e., the Bergman fan of the free matroid $U_{n,n}$. As noted in the discussion right after \cite[Proposition 2.1.7]{backman2023simplicial},
the inclusion of fans $\iota: \Sigma_M \hookrightarrow \Sigma_{U_{n,n}}$ induces the surjective pullback map
\begin{align}
    \iota^*:A^*(U_{n,n}) &\longrightarrow A^*(M) \label{eq:pullback}\\
    x_S &\longmapsto 
    \begin{cases}
        x_S &\text{if } S\subseteq E \text{ is a flat of } M,\nonumber\\
        0 &\text{otherwise}\nonumber. 
    \end{cases}\nonumber
\end{align}
Any Minkowski weight $\omega \in \MW_k(\Sigma_M)$ can also be seen as a Minkowski weight $\iota_*(\omega) \in \MW_k(\Sigma_{U_{n,n}})$. If $\xi \in A^l({U_{n,n}})$ and $\omega \in \MW_{k}(\Sigma_M)$, we then have 
$$\xi \cap \iota_*(\omega) = \iota_*(\iota^*(\xi) \cap \omega).$$

\subsection{Tautological classes of matroids} 
\label{subsec:taut}
In the paper \cite{berget2023tautological}, Berget, Eur, Spink, and Tseng introduced the tautological classes of a matroid. This powerful framework allowed for the unification of various results from several separate geometric models for matroids. 

Let $E=[n]$, and consider the $(n-1)$-dimensional permutahedral toric variety $X_E$, which is the projective toric variety associated to the permutahedral fan $\Sigma_{U_{n,n}}$. 
Denote by $K_0^T(X_E)$ be the $T$-equivariant $K$-ring of $X_E$.
The Chow ring of $X_E$ is the Chow ring $A^*(U_{n,n})$ of the free matroid $U_{n,n}$. 
Suppose $M$ is a matroid on the set $E$. In the case $M$ is realizable by a linear subspace $L \subseteq \mathbb{C}^E$, the $T$-equivariant $K$-classes $[S_L],[Q_L]\in K_0^T(X_E)$ of the tautological subbundle $S_L$ and quotient bundle $Q_L$ on $X_E$ (see \cite[Definition 1.2]{berget2023tautological}) depend only on the matroid $M$. This leads to a combinatorial definition of the \textit{tautological $K$-classes} of any matroid $M$, which are elements $[S_M],[Q_M] \in K_0^T(X_E)$ (see \cite[Definition 3.9]{berget2023tautological}). The \textit{tautological Chern classes} $c_i(S_M), c_i(Q_M)\in A^i(U_{n,n})$ of a matroid $M$ are the non-equivariant Chern classes of $[S_M],[Q_M]$.

A formula for the Chern roots of $[S_M]$ and $[Q_M]$ is provided in \cite[Remark III.1]{berget2023tautological}, as we now explain.
Fix a sequence of matroids $M_0, M_1, \dots, M_{n}$ on the set $[n]$ such that $M_i$ has rank $i$, the matroid $M_r=M$, and $M_i$ is a quotient of $M_j$ for every $i<j$. Such a sequence has the property that $M_0 = U_{0,n}$ and $M_n=U_{n,n}$. Denote by $\ranki$ the rank function of $M_i$. 
Then, we have the following expressions for the Chern polynomials of $[S_M]$ and $[Q_M]$:
        \begin{align} 
            \label{ChernEquS}
    \sum_{k=0}^{r} c_k(S_M)t^k&= \prod_{i=0}^{r-1}\Biggl(1-t\sum_{\emptyset\subsetneq S\subseteq E}\bigl(\rankiii(S)-\ranki(S) \bigr)x_S  \Biggr) \quad \in A^*(U_{n,n})[t],\\
            \label{ChernEquQ}
    \sum_{k=0}^{n-r} c_k(Q_M)t^k&= \prod_{i=r}^{n-1}\Biggl(1-t\sum_{\emptyset\subsetneq S\subseteq E}\bigl(\rankiii(S)-\ranki(S) \bigr)x_S  \Biggr) \quad \in A^*(U_{n,n})[t].
        \end{align}

\subsection{Chern-Schwartz-MacPherson (CSM) cycles of matroids.} López de {Me\-dra\-no}, Rincón, and Shaw introduced in \cite{de2020chern} certain Minkowski weights ${\rm csm}_k(M)$ on the Bergman fan $\Sigma_M$ of a rank-$r$ matroid $M$, one for each dimension $0\leq k \leq r-1$, called the \textit{Chern-Schwartz-MacPherson (CSM) cycles} of $M$. The weights of the maximal cones of ${\rm csm}_k(M)$ are given by products of beta invariants of certain minors of $M$; see \cite[Definition 2.8]{de2020chern} for details. This construction is motivated by the fact that, when the matroid $M$ arises from a complex hyperplane arrangement, these Minkowski weights encode the Chern-Schwartz-MacPherson classes of the complement of the arrangement inside its wonderful compactification.

As shown in \cite{berget2023tautological}, CSM cycles of matroids are one of the various invariants that are nicely obtained from their tautological Chern classes.  Concretely, given a matroid $M$ of rank $r$, \cite[Theorem C]{berget2023tautological} states that
     \begin{equation} 
     \label{eq:C}
         {\rm csm}_k(M)=\iota^*(c_{r-1-k}(S_M))\cap \omega_M  \quad \in \MW_k(\Sigma_M),
     \end{equation}
where $\iota^*$ is as in \cref{eq:pullback} and $\omega_M$ is the Bergman class of $\Sigma_M$.
Keeping in mind the isomorphism in \cref{eq:isomorphism}, this shows that 
\begin{equation}\label{eq:ch}
{\rm ch}_k(M) := \iota^*(c_{k}(S_M)) \quad \in A^k(M)
\end{equation}
is the element of $A^k(M)$ satisfying 
\begin{equation}\label{eq:dualitycsm}
\delta_{\Sigma_M}({\rm ch}_k(M)) = {\rm csm}_{r-1-k}(M).
\end{equation}

\section{The rank-nullity ring of a matroid}\label{sec:rank-nullity}

\subsection{A formula for the tautological Chern classes}
\label{subsec:formula}
In this subsection we give an explicit formula for the Chern classes $c_k(S_M)$ and $c_k(Q_M)$ of the tautological subbundle and  quotient bundle of a matroid $M$ in the presentation given in \cref{def:chow} of the Chow ring of the free matroid $A^*(U_{n,n})$. 
We denote by $\rank$ the rank function of $M$, and by $\nulll$ its nullity function, namely, for $S\subseteq E$,
$$\nulll(S)=|S|-\rank(S).$$

\begin{thm} \label{mainformula}
Let $M$ be a matroid of rank $r$ on the set $E=[n]$. For $0\leq k \leq r$, the Chern class $c_k(S_M)$ of the tautological subbundle $S_M$ is given by 
$$c_k(S_M)=\sum \alpha_S(S_1, \dots, S_t, p_1, \dots, p_t) \, x_{S_1}^{p_1}x_{S_2}^{p_2}\ldots x_{S_t}^{p_t} \quad \in A^k(U_{n,n}),
$$
where the sum is taken over all chains of subsets $\emptyset \subsetneq S_1 \subsetneq \cdots \subsetneq S_t \subseteq E$ and all positive integers $p_1, \dots,p_t$ satisfying $\sum_{i=1}^t p_i = k$, with coefficients given by
$$\alpha_S(S_1, \dots, S_t, p_1, \dots, p_t) =
(-1)^k\binom{\rank(S_1)}{p_1}
\binom{\rank(S_2)-\Tilde{p}_1}{p_2}\dots
\binom{\rank(S_t)-\Tilde{p}_{t-1}}{p_t}$$
where $\Tilde{p}_i=\sum_{j=1}^ip_j$.
Similarly, for any $0\leq k \leq n-r$, the Chern class $c_k(Q_M)$ of the tautological quotient bundle $Q_M$ is given by 
$$c_k(Q_M) = \sum \alpha_Q(S_1, \dots, S_t, p_1, \dots, p_t) \, x_{S_1}^{p_1}x_{S_2}^{p_2}\ldots x_{S_t}^{p_t} \quad \in A^k(U_{n,n}),$$
where the sum is taken over all chains of subsets $\emptyset \subsetneq S_1 \subsetneq \cdots \subsetneq S_t \subseteq E$ and all positive integers $p_1, \dots, p_t$ satisfying $\sum_{i=1}^t p_i = k$, with coefficients given by
$$\alpha_Q(S_1, \dots, S_t, p_1, \dots, p_t) =
(-1)^k\binom{\nulll(S_1)}{p_1}
\binom{\nulll(S_2)-\Tilde{p}_1}{p_2}\dots
\binom{\nulll(S_t)-\Tilde{p}_{t-1}}{p_t}$$
where $\Tilde{p}_i=\sum_{j=1}^ip_j$.
\end{thm}

\begin{proof}
Given a sequence of matroids $M_0, M_1, \ldots, M_{n}$ with properties as given in \cref{subsec:taut}, we use the formula for the Chern roots of $S_M$ and $Q_M$ in \cite[Remark III.1]{berget2023tautological}, recalled above in \cref{ChernEquS} and \cref{ChernEquQ}: 
\begin{equation}\label{ChernEquSagain} 
    \sum_{k=0}^{r} c_k(S_M)t^k= \prod_{i=0}^{r-1}\Biggl(1-t\sum_{\emptyset\subsetneq S\subseteq E}\bigl(\rankiii(S)-\ranki(S) \bigr)x_S  \Biggr),
\end{equation}
\begin{equation*} 
    \sum_{k=0}^{n-r} c_k(Q_M)t^k= \prod_{i=r}^{n-1}\Biggl(1-t\sum_{\emptyset\subsetneq S\subseteq E}\bigl(\rankiii(S)-\ranki(S) \bigr)x_S  \Biggr).
\end{equation*}
An explicit formula for each of the Chern classes $c_k(S_M)$ and $c_k(Q_M)$ is obtained by expanding the products above. The fact that the sequence of matroids $M_0, M_1, \dots, M_n$ increases rank by one every step implies that the difference $\rankiii(S)-\ranki(S)$ is either $0$ or $1$ for all $i$ and all $S \subseteq [n]$. For each $0 \leq i \leq n-1$, let
$$\mathscr{S}_i=\{S\subseteq E: \rankiii(S)-\ranki(S)=1\}.$$
Moreover, for each nonempty $S \subseteq E$, let 
\begin{align*}
  \zeta_S^{S_M}&=\{i\in[0,r-1]: S\in\mathscr{S}_i\},\\ 
  \zeta_S^{Q_M}&=\{i\in[r,n-1]: S\in\mathscr{S}_i\}.
\end{align*}
Whenever $S\subseteq T$ we have the inclusions $\zeta^{S_M}_S\subseteq \zeta_T^{S_M}$ and $\zeta^{Q_M}_S\subseteq \zeta_T^{Q_M}$; see \cite[Proposition 7.3.6 (iii)]{oxley2006matroid}. Moreover, note that $|\zeta_S^{S_M}|= \rank(S)$ since $\rankMz(S)= 0$ and $\rankMr(S)=\rank(S)$. 
Likewise, we have $|\zeta_S^{Q_M}|= \nulll(S)$ since $\rankMn(S)= |S|$ and $\rankMr(S)=\rank(S)$. 

To prove the formula for $c_k(S_M)$, we write the product in \cref{ChernEquSagain} as  
\begin{equation*} 
\label{ChernEqu2}
    \prod_{i=0}^{r-1}\big(1-t \sum_{S\in \mathscr{S}_i} x_S\big).
\end{equation*}
By the relations in the Chow ring $A^*(U_{n,n})$, if $S$ and $T$ are incomparable then $x_Sx_T=0$. 
When expanding this product, we thus obtain only monomials of the form $x_{S_1}^{p_1}x_{S_2}^{p_2}\ldots x_{S_t}^{p_t}$ for a chain of subsets $\emptyset \subsetneq S_1 \subsetneq \cdots \subsetneq S_t \subseteq E$. 
To find the corresponding coefficient $\alpha_S(S_1, \dots, S_t, p_1, \dots, p_t)$, we need to count the number of times the monomial $x_{S_1}^{p_1}x_{S_2}^{p_2}\ldots x_{S_t}^{p_t}$ is obtained in this expansion. The variable $x_{S_1}$ appears in exactly $|\zeta^{S_M}_{S_1}|=\rank(S_1)$ many of the sums $\sum_{S\in \mathscr{S}_i} x_S$, hence there are $\binom{\rank(S_1)}{p_1}$ number of ways of obtaining $x_{S_1}^{p_1}$. The variable $x_{S_2}$ appears in exactly $|\zeta^{S_M}_{S_2}|=\rank(S_2)$ many of the sums $\sum_{S\in \mathscr{S}_i} x_S$ (which include all the sums that contained $x_{S_1}$), but we cannot choose it from the sums where we chose $x_{S_1}$, hence there are $\binom{\rank(S_2)-p_1}{p_2}$ ways of obtaining $x_{S_2}^{p_2}$. 
In general, the variable $x_{S_j}$ appears in exactly $|\zeta^{S_M}_{S_j}|=\rank(S_j)$ many sums $\sum_{S\in \mathscr{S}_i} x_S$  (which include all the sums that contained the previous variables), and since we cannot choose it from the sums where we have chosen previous variables $x_{S_l}$ with $1 \leq l\leq j-1$, there are $\binom{\rank(S_j)-\Tilde{p}_{j-1}}{p_j}$ ways of obtaining $x_{S_j}^{p_j}$. 
This proves the claimed formula for $c_k(S_M)$. A similar argument proves the formula for $c_k(Q_M)$.
\end{proof}

\begin{example}
\label[example]{ex:K4}
Consider the rank-$3$ graphic matroid $M$ arising from the complete graph $K_4$ on 4 vertices. The ground set of $M$ is the set of edges of $K_4$; see \cref{im: K4}. The bases of $M$ are the spanning trees of $K_4$, so the subsets of size 3 that are not bases are $\{\{1,2,5\},\{1,4,6\},\{2,3,6\},\{3,4,5\}\}$. 
\begin{figure}[htb]
\begin{center}
\begin{tikzpicture}[line width=1.2pt, scale=2.5]
\tikzset{
    v/.style={circle, fill=black, inner sep=1.2pt},
    lab/.style={fill=white, inner sep=2pt}
  }

  \node[v]  (A) at (0,1) {};
  \node[v]  (B) at (1,1) {};
  \node[v]  (C) at (1,0) {};
  \node[v]  (D) at (0,0) {};

  \draw (A) -- (B) node[midway, above] {$1$};
  \draw (B) -- (C) node[midway, right] {$2$};
  \draw (C) -- (D) node[midway, below] {$3$};
  \draw (D) -- (A) node[midway, left] {$4$};
  \draw (D) -- (B) node[midway, above = 12 pt, , xshift=-12pt] {$5$};
  \draw (A) -- (C) node[midway, above = 12 pt, xshift=12pt] {$6$};
\end{tikzpicture}
    \caption{The complete graph $K_4$.}
        \label{im: K4}
\end{center}
\end{figure}
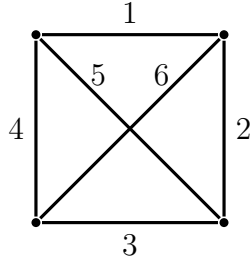
The sequence $(U_{0,6}, U_{1,6}, U_{2,6}, M, U_{4,6}, U_{5,6}, U_{6,6})$
satisfies the properties given in the proof of \cref{mainformula}. The corresponding subsets $\mathscr{S}_i$ are:
\begin{align*}
       \mathscr{S}_0&=\{\emptyset \subsetneq S\subseteq E \},\ 
       &\mathscr{S}_3&=\{\emptyset \subsetneq S\subseteq E : \nulll(S) \geq 1 \},\\
       \mathscr{S}_1&=\{\emptyset \subsetneq S\subseteq E : \rank(S) \geq 2 \},\ 
       &\mathscr{S}_4&=\{\emptyset \subsetneq S\subseteq E : \nulll(S) \geq 2 \},\\
       \mathscr{S}_2&=\{\emptyset \subsetneq S\subseteq E : \rank(S) =3 \},\ &\mathscr{S}_5&=\{E \}. 
   \end{align*}
   Moreover, we have that
    \begin{equation*}
       \zeta_S^{S_M} = 
       \begin{cases}
       \{0\} &\text{ if } \rank(S) = 1,\\
       \{0,1\} &\text{ if } \rank(S) = 2,\\
        \{0,1,2\} &\text{ if } \rank(S) = 3,\\ 
        \end{cases} \ \qquad \ 
        \zeta_S^{Q_M} = \begin{cases}
       \{3\} &\text{ if } \nulll(S) = 1,\\
       \{3,4\} &\text{ if } \nulll(S) = 2, \\
       \{3,4,5\} &\text{ if } S = E.     
       \end{cases}
   \end{equation*}
 Note that $\zeta_S^{S_M} \subseteq \zeta_T^{S_M}$ and $\zeta_S^{Q_M} \subseteq \zeta_T^{Q_M}$  whenever $S\subseteq T$, as predicted. 
 According to Equations \eqref{ChernEquS} and \eqref{ChernEquQ}, the Chern classes of $S_{M}$ and $Q_{M}$ are given by 
 \begin{align*}
 \sum_{i=0}^3 c_i(S_M)t^i&= (1-t\sum_{S\in \mathscr{S}_0}x_S)(1-t\sum_{S\in \mathscr{S}_1}x_S)(1-t\sum_{S\in \mathscr{S}_2}x_S), \\
  \sum_{i=0}^3 c_i(Q_M)t^i&= (1-t\sum_{S\in \mathscr{S}_3}x_S)(1-t\sum_{S\in \mathscr{S}_4}x_S)(1-t\sum_{S\in \mathscr{S}_5}x_S).
 \end{align*}
For $i = 1,2,3$, let $\mathcal T_i$ be the set of nonempty subsets of $E$ of rank $i$. The expression for the Chern class $c_2(S_M)$ given in \cref{mainformula} is 
\begin{align*}
    c_2(S_M)= \sum_{S\in \mathcal{T}_2} x_S^2 
    & + 3 \sum_{S\in \mathcal{T}_3} x_S^2 
+ \sum_{S\in \mathcal{T}_1,T\in \mathcal{T}_2}x_{S}x_{T}
+ 2 \sum_{S\in \mathcal{T}_1,T\in \mathcal{T}_3}x_{S}x_{T}\\ 
    & + 2\sum_{S,T\in \mathcal{T}_2,\  S\subsetneq T}x_{S}x_{T} 
+ 4 \sum_{S\in \mathcal{T}_2,T\in \mathcal{T}_3}x_{S}x_{T}
    +  6 \sum_{S,T\in \mathcal{T}_3,\ S\subsetneq T}x_{S}x_{T},
\end{align*}
where all these sums can be taken only over subsets that form a chain under inclusion.
\end{example}

We now give a formula for ${\rm ch}_k(M) \in A^k(M)$, which is the dual element to the csm cycle ${\rm csm}_{r-1-k}(M)\in \MW_{r-1-k}(M)$, as in \cref{eq:dualitycsm}. Inspired by a previous version of this paper in which our formula was still conjectural, the authors of \cite{caraballo2024staircase} have also provided a separate proof of an equivalent statement (see \cref{rem:staircase}), including a geometric proof in the realizable case. 

\begin{cor}
\label[cor]{corcsm}
Let $M$ be a matroid of rank $r$ on $E=[n]$. Then the element ${\rm ch}_k(M) \in A^k(M)$ dual to ${\rm csm}_{r-1-k}(M)\in \MW_{r-1-k}(M)$ is given by
$${\rm ch}_k(M)=\sum \alpha_S(F_1 , \dots, F_t,p_1,\dots,p_t)\,x_{F_1}^{p_1}x_{F_2}^{p_2}\ldots x_{F_t}^{p_t},$$
where the sum is taken over all chains of flats $\emptyset \subsetneq F_1 \subsetneq \cdots \subsetneq F_t \subseteq E$ and all positive integers $p_1, \dots, p_t$ satisfying $\sum_{i=1}^t p_i = k$, with coefficients given by
$$\alpha_S(F_1 , \ldots, F_t,p_1,\dots,p_t) =
(-1)^k\binom{\rank(F_1)}{p_1}
\binom{\rank(F_2)-\Tilde{p}_1}{p_2}\ldots
\binom{\rank(F_t)-\Tilde{p}_{t-1}}{p_t}$$
where $\Tilde{p}_i=\sum_{j=1}^ip_j$.
\end{cor}
\begin{proof}
\cref{eq:ch} states that ${\rm ch}_k(M) = \iota^*(c_k(S_M))$. The result thus follows from \cref{mainformula} and the description of $\iota^*$ given in \cref{eq:pullback}. 
\end{proof}

\begin{rem}\label[rem]{rem:staircase} 
One possible choice for the sequence of matroids $M_0, M_1, \dots, M_n$ in the proof of \cref{mainformula} is 
the \emph{full Higgs lift} of the matroid $M$, where $M_i$ is taken to be the matroid with bases
\begin{align*}
    \mathcal{B}_i=\Big\{ S\in \binom{E}{i}\ :\ S \text{ contains or is contained in a basis of }M\Big\}. 
\end{align*}
For example, the sequence given in \cref{ex:K4} is the full Higgs lift of the graphic matroid of $K_4$.
With this choice, \cref{ChernEquS} and \cref{ChernEquQ} become
\begin{align}
    \sum_{k=0}^r c_k(S_M)t^k=& \prod_{i=1}^{r}\Big(1-\sum_{\rank(S)\geq i}x_S\Big)t,\label{eq1staircase}\\
    \sum_{k=0}^{n-r} c_k(Q_M)t^k=&\prod_{i=1}^{n-r}\Big(1-\sum_{\nulll(S)\geq i}x_S\Big)t.
\end{align}
Indeed, for $0 \leq i \leq r-1$, the rank function $\ranki$ of the matroid $M_i$ is given by $\ranki(S) = \min(i, \rank(S))$, which implies that
\begin{align*}
    \mathscr{S}_{i}=\{S\subseteq E: \rankiii(S)-\ranki(S)=1\}=\{S\subseteq E : \rank(S)\geq i+1 \}.
\end{align*}
Similarly, for $r+1 \leq i\leq n$, the rank function $\ranki$ of the matroid $M_i$ is given by
$\ranki(S)=\min(|S|,\rank(S)+i-r)$, which implies that 
\begin{align*}
    \mathscr{S}_{r+j} = \{S\subseteq E: r_{M_{r+j+1}}(S)-r_{M_{r+j}}(S)=1\}=\{S\subseteq E \ |\  \nulll(S)\geq j +1\}.
\end{align*}

After taking the cap product of \cref{eq1staircase} with the Bergman class $\omega_M$ of $M$, this leads to the ``staircase'' formula for csm cycles of $M$ provided in \cite{caraballo2024staircase}:
\[\sum_{k=0}^r {\rm ch}_k(M)\,t^k= \prod_{i=1}^{r}\Big(1-\sum_{\substack{F \text{ flat}\\\rank(F)\geq i}}x_F\Big)\,t \quad \in A^*(M)[t].\]
This formula is equivalent to our \cref{corcsm}, by a similar argument to the proof of \cref{mainformula}.
\end{rem}

\subsection{The rank-nullity ring of a matroid}
In this subsection we introduce the rank-nullity ring of a matroid $M$ as certain subring of the permutahedral Chow ring $A^*(U_{n,n})$, and show that it contains the tautological Chern classes $c_k(S_M),c_k(Q_M)$ of $M$.
\begin{defn}
\label{def:ranknullity}
Let $M$ be a matroid on the set $[n]$ with rank function $\rank$ and nullity function $\nulll$. 
For $0\leq i,j \leq n$, consider the degree-$1$ element
$$y_{i,j} := \sum_{\substack{\emptyset \neq S \subseteq [n] \\ \rank(S)=i,\ \nulll(S)=j}}x_S \quad \in A^*(U_{n,n}).$$
Note that $y_{i,j}=0$ whenever $i+j > n$ or $(i,j) = (0,0)$.
The \emph{rank-nullity ring} $R^*(M)$ of the matroid $M$ is the graded subring
$$R^*(M) := \mathbb{Z}[y_{i,j} \,:\, 0 \leq i, j \leq n] \quad \subseteq A^*(U_{n,n}).$$
Note that we do not require the matroid $M$ to be loopless.
\end{defn}

We provide a simple example of a rank-nullity ring before stating our main result of this subsection.

\begin{example}
\label[example]{ex:ranknullityK4}
Consider the rank-$3$ graphic matroid $M$ of the complete graph $K_4$ on 4 vertices, discussed in \cref{ex:K4}. 
The rank-nullity ring $R^*(M)$ is equal to 
    $$R^*(M):=\mathbb{Z}[y_{1,0},y_{2,0},y_{2,1},y_{3,0},y_{3,1},  
    y_{3,2}, y_{3,3}] \quad \subseteq A^*(U_{6,6})$$
where the elements $y_{i,j}$ are 
\begin{align*}
    y_{1,0} &= \sum_{1\leq i\leq 6} x_{\{i\}},
    &y_{2,0} &= \sum_{1\leq i<j \leq 6} x_{\{i,j\}}, \\
    y_{2,1} &= x_{\{1,2,5\}}+x_{\{1,4,6\}}+x_{\{2,3,6\}}+x_{\{3,4,5\}},
    &y_{3,0} &= \sum_{1\leq i<j<k \leq 6} x_{\{i,j,k\}} - y_{2,1} \\
 y_{3,1} &= \sum_{1\leq i<j<k<l \leq 6} x_{\{i,j,k,l\}} ,
  & y_{3,2} &= \sum_{1\leq i<j<k<l<m \leq 6} x_{\{i,j,k,l,m\}}, \\
   y_{3,3} &=  x_{\{1,2,3,4,5,6\}}.
\end{align*}
The subring $R^*(M)$ is generated by only $7$ elements, compared to the $63$ variables that generate $A^*(U_{6,6})$.
We will see in \cref{ex:hilbertfunctions} that the degree-$1$ part of $R^*(M)$ is isomorphic to $\ZZ^6$, while the degree-$1$ part of $A^*(U_{6,6})$ is isomorphic to $\ZZ^{57}$.   
\end{example}

Our motivation for introducing the rank-nullity ring of a matroid comes from the following fact.
\begin{thm}\label{thmy}
    Let $M$ be a rank-$r$ matroid on the set $[n]$ with rank function $\rank$ and nullity function $\nulll$. The Chern classes $c_k(S_M)$ and $c_k(Q_M)$ of the tautological bundles of $M$ belong to the rank-nullity ring $R^*(M)$. More concretely, 
    \begin{align*}
        c_k(S_M)&=\sum \alpha_S(r_1,\ldots, r_t, p_1, \dots, p_t)\, y_{r_1,n_1}^{p_1}y_{r_2,n_2}^{p_2}\cdots y_{r_t,n_t}^{p_t}\quad \text{and}\\
        c_k(Q_M)&=\sum \alpha_Q(n_1,\ldots, n_t, p_1, \dots, p_t)\, y_{r_1,n_1}^{p_1}y_{r_2,n_2}^{p_2}\cdots y_{r_t,n_t}^{p_t},
    \end{align*}
    where the sums are taken over all sequences $(0,0)<(r_1,n_1)<  \dots < (r_t, n_t) \leq (r,n-r)$ (under the coordinatewise partial order) and all positive integers $p_1, \dots, p_t$ satisfying $\sum_{m=1}^tp_m = k$, with coefficients given by 
    \begin{align*}
        \alpha_S(r_1,\ldots, r_t, p_1, \dots, p_t) &:= 
        (-1)^k\binom{r_1}{p_1} \binom{r_2-\Tilde{p}_1}{p_2}\dots \binom{r_t-\Tilde{p}_{t-1}}{p_t}\quad \text{and}\\
        \alpha_Q(n_1,\ldots, n_t, p_1, \ldots, p_t) &:= (-1)^k\binom{n_1}{p_1}\binom{n_2-\Tilde{p}_1}{p_2}\dots\binom{n_t-\Tilde{p}_{t-1}}{p_t}, 
    \end{align*}
    where $\Tilde{p}_i=\sum_{j=1}^ip_j$.
\end{thm}

\begin{proof}    
For fixed $0\leq i,j \leq n$, any two distinct subsets $S, T \subseteq [n]$ of rank $i$ and nullity $j$ have both size $i+j$ and are thus incomparable, which implies that $x_Sx_T=0$ in $A^*(U_{n,n})$. It follows that
\begin{equation}
\label{eq: prod}
y_{i,j}^m=\sum_{\rank(S)=i,\, \nulll(S)=j}x_S^m
\end{equation}
 for all $m \geq 0$. 
Recall the formula 
\begin{equation}\label{eq:chernS}
c_k(S_M)=\sum_{\substack{\emptyset \subsetneq S_1 \subsetneq \cdots \subsetneq S_t \subseteq [n]\\p_1+\dots+p_t = k}} \alpha_S(S_1, \dots, S_t, p_1, \dots, p_t) \, x_{S_1}^{p_1}x_{S_2}^{p_2}\ldots x_{S_t}^{p_t}
\end{equation}
stated in \cref{mainformula}.
Note that the coefficients $\alpha_S(S_1, \dots, S_t, p_1, \dots, p_t)$ given there only depend on the sequence of ranks $\rank(S_1), \dots, \rank(S_t)$ and the exponents $p_1, \dots, p_t$. Abusing notation slightly, we hence denote
\begin{align*}
  \alpha_S(S_1, \dots, S_t, p_1, \dots, p_t) &=\alpha_S(\rank(S_1), \dots, \rank(S_t), p_1, \dots, p_t).
\end{align*}
We can group together the terms in \cref{eq:chernS} corresponding to the same sequence $(0,0)<(r_1,n_1)<  \dots < (r_t, n_t) \leq (r,n-r)$ of ranks and nullities to obtain 
$$c_k(S_M) = \sum_{\substack{(r_1,n_1)<  \dots < (r_t, n_t)\\p_1+\dots+p_t = k}} 
    \bigl(  \alpha_S(r_1,\ldots, r_t, p_1, \ldots, p_t) \sum_{\substack{S_1\subsetneq \dots \subsetneq S_t\\
    \rank(S_j)=r_j\\
    \nulll(S_j)=n_j}} x_{S_1}^{p_1}x_{S_2}^{p_2} \dots x_{S_t}^{p_t} \bigr).$$
In the inner sum we can ignore the condition that the subsets $S_1, \dots, S_t$ form a chain (since the extra terms are all equal to $0$), and factor each summand as  
\begin{align*}
 c_k(S_M) &= \sum_{\substack{(r_1,n_1)<  \dots < (r_t, n_t)\\p_1+\dots+p_t = k}} \alpha_S(r_1,\ldots, r_t, p_1, \ldots p_t) \bigl( \sum_{\substack{
    \rank(S_1)=r_1\\
    \nulll(S_1)=n_1}} \!\!x_{S_1}^{p_1} \bigr) \bigl( \sum_{\substack{
    \rank(S_2)=r_2\\
    \nulll(S_2)=n_2}} \!\!x_{S_2}^{p_2} \bigr) \cdots \bigl( \sum_{\substack{
    \rank(S_t)=r_t\\
    \nulll(S_t)=n_t}} \!\! x_{S_t}^{p_t} \bigr) \\ 
    &=  \sum_{\substack{(r_1,n_1)<  \dots < (r_t, n_t)\\p_1+\dots+p_t = k}}  \alpha_S(r_1,\ldots, r_t, p_1, \ldots p_t) \, y_{r_1,n_1}^{p_1}y_{r_2,n_2}^{p_2}\cdots y_{r_t,n_t}^{p_t},
\end{align*}
which is the desired formula for $c_k(S_M)$. A similar argument proves the claimed formula for $c_k(Q_M)$.
\end{proof}

The rank-nullity ring $R^*(M)$ of a matroid $M$ can be much smaller than the Chow ring $A^*(U_{n,n})$. The following result quantifies exactly how many linearly independent generators (of degree 1) this subring has.

\begin{prop}
\label[prop]{prop:rel}Let $M$ be a matroid of rank $r$ on the set $[n]$. If $M$ is not a direct sum of a free matroid and a rank-$0$ matroid (both having at least one element), all linear relations among the generators $y_{i,j}$ of $R^*(M)\subseteq A^*(U_{n,n})$, where $0\leq i,j\leq n$, are generated by the following ones:
$$y_{0,0} = 0,$$
\vspace{-2mm}
$$y_{i,j} = 0 \quad \text{whenever } i+j > \max\{ |F| : F \text{ a rank-$i$ flat}\},$$ 
\begin{align*}
        \sum_{0\leq i,j \leq n} (i+j) \, y_{i,j} = 0.
    \end{align*}
If $M \cong U_{r,r} \oplus U_{0,n-r}$ with $1 \leq r \leq n-1$ then the same statement is true but replacing the last relation by the following two:
    \begin{align*}
        \sum_{0\leq i,j \leq n} i \, y_{i,j} = 0 \qquad \text{ and } \qquad \sum_{0\leq i,j \leq n} j \, y_{i,j} = 0.
    \end{align*}
\end{prop}
\begin{proof}
   We first show that all these relations hold. By definition, $y_{i,j} = 0$ whenever there are no subsets of $[n]$ of rank $i$ and nullity $j$. Note that, for fixed $i$, the set of numbers $j$ for which this is the case is closed upwards, as it is always possible to remove an element from a dependent subset $S \subseteq [n]$ that decreases its nullity but not its rank. The rank-$i$ subsets of maximum nullity are flats of rank $i$, which shows the first set of relations hold. 
   
   The last relation is obtained by adding over all $e \in [n]$ the linear relations $0 = \sum_{S \ni e} x_S$ that hold in the Chow ring $A^*(U_{n,n})$, namely
   \begin{equation}\label{eq:onerelation}
       0 = \sum_{e \in [n]} \sum_{S \ni e} x_S = \sum_{\emptyset \neq S \subseteq [n]} |S| \, x_S = \sum_{0\leq i,j \leq n} (i+j) \, y_{i,j}.
   \end{equation}
   If $M$ consists only of coloops and loops, this last relation is actually implied by the following two:
   \begin{equation}\label{eq:tworelations}
       0 =  \sum_{e \text{ coloop}} \sum_{S \ni e} x_S = \sum_{0\leq i,j \leq n} i \, y_{i,j}  \qquad \text{ and } \qquad \sum_{e \text{ loop}} \sum_{S \ni e} x_S = \sum_{0\leq i,j \leq n} j \, y_{i,j}.
   \end{equation}
   
   To see that these relations generate all linear relations, suppose $\sum_{0\leq i,j \leq n} \alpha_{i,j} \, y_{i,j} = 0$
   for integers $\alpha_{i,j}$. After removing all the $y_{i,j}$ for which there are no subsets of $[n]$ of rank $i$ and nullity $j$, we can assume that all the $y_{i,j}$ involved in this linear relation are non-zero linear combinations of the variables $x_S$ of $A^*(U_{n,n})$. This relation must be generated by the linear relations that hold in $A^*(U_{n,n})$, so there are integers $\beta_e$ for $e \in [n]$ satisfying
   \begin{equation}\label{eq:linearequation}
      \sum_{0\leq i,j \leq n} \alpha_{i,j} \, y_{i,j} = \sum_{e \in [n]} \left( \beta_e \, \sum_{S \ni e} x_S \right).
   \end{equation}
   For any $e \in [n]$ that is not a loop, the coefficient of $x_{\{e\}}$ on the left-hand side of \cref{eq:linearequation} is $\alpha_{1,0}$, while on the right-hand side it is $\beta_e$; this shows that all coefficients $\beta_e$ with $e$ not a loop are equal to $\alpha_{1,0}$. 
   Similarly, for any loop $e \in [n]$, the coefficient of $x_{\{e\}}$ on the left-hand side of \cref{eq:linearequation} is $\alpha_{0,1}$, while on the right-hand side it is $\beta_e$, and so all $\beta_e$ with $e$ a loop are equal to $\alpha_{0,1}$.
   If $M$ consists only of coloops and loops, this shows that the linear relation is a combination of the two relations in \cref{eq:tworelations}.
   If $M$ has no loops, all the coefficients $\beta_e$ are equal, so the relation is a multiple of the relation in \cref{eq:onerelation}.
   Finally, if $M$ has a loop but does not decompose as $M \cong U_{r,r} \oplus U_{0,n-r}$ with $1 \leq r \leq n-1$, there must exist a circuit $C$ of $M$ of size $k \geq 2$. Take $c\in C$ and let $\ell$ be a loop of $M$. The coefficient of $x_{C}$ on the left-hand side of \cref{eq:linearequation} is $\alpha_{k-1,1}$, while on the right-hand side it is $k \alpha_{1,0}$. Furthermore, the coefficient of $x_{(C-c)\cup \ell}$ on the left-hand side of \cref{eq:linearequation} is $\alpha_{k-1,1}$, while on the right-hand side it is $(k-1) \alpha_{1,0} + \alpha_{0,1}$. This implies that $\alpha_{1,0} = \alpha_{0,1}$, and so the linear relation is again a multiple of the relation in \cref{eq:onerelation}.
\end{proof}

We conclude this subsection with a remark on the symmetry of the rank-nullity ring of a matroid.

The \emph{automorphism group} $\Aut(M)$ of a matroid $M$ on a ground set $E$ is the group of permutations of $E$ that preserve the rank of every subset.
For example, the automorphism group $\Aut(U_{r,n})$ of the uniform matroid $U_{r,n}$ is the symmetric group $S_n$. 
The group $\Aut(M)$ acts naturally on the Chow ring $A^*(M)$ of $M$; namely, a permutation $\sigma \in \Aut(M)$ acts as $\sigma(x_S)=x_{\sigma(S)}$.

\begin{prop}
\label[prop]{propaut}
Let $M$ be a matroid on the set $[n]$. The rank-nullity subring $R^*(M) \subseteq A^*(U_{n,n})$ is $\Aut(M)$-invariant, i.e., $$R^*(M) \subseteq A^*(U_{n,n})^{\Aut(M)}.$$ 
\end{prop}
\begin{proof}
    Any $\sigma \in \Aut(M)$ permutes the set of subsets of $[n]$ of any given rank and nullity, and thus it fixes each of the generators $y_{i,j}\in R^*(M)$.
\end{proof}

The following example shows that the inclusion $R^*(M)\subseteq A^*(U_{n,n})^{\Aut(M)}$ can be strict.
\begin{example}
Again, consider the rank-$3$ graphic matroid $M$ of the complete graph $K_4$ on 4 vertices, discussed in \cref{ex:K4}. In this case, the automorphism group of $M$ is induced from the automorphism group of the underlying graph $K_4$, which is isomorphic to the symmetric group $S_4$. Explicitly, the automorphism group of $M$ consists of the permutations (listed in cycle notation) 
     \begin{align*}
         \Aut(M)=\{&(),\ (1 6)(3 5),\ (14)(23),\ (45)(26),\ (46)(52),\ (12)(34),\  (63)(15),\\ & (164)(235),\ (146)(253),\ (126)(345),\ (162)(354),\ (154)(362),\ (145)(326),\\ 
         & (125)(346),\ (152)(364),\ (65)(13),\ (13)(24),\ (24)(56),\ (1234)(56),\\ 
         &\ (13)(2645),\ (1536)(24),\ (13)(2546),\ (1635)(24), (1432)(56)\}\simeq S_4. 
     \end{align*}
Since any automorphism permutes the three perfect matchings of $K_4$, we have
$$x_{\{1,3\}}+x_{\{2,4\}}+x_{\{5,6\}}\in A^*(U_{n,n})^{\text{Aut}(M)}.$$
We can see from the description of the rank-nullity ring of $M$ given in \cref{ex:ranknullityK4} that $x_{\{1,3\}}+x_{\{2,4\}}+x_{\{5,6\}}\notin R^*(M)$, so the inclusion $R^*(M)\subseteq A^*(U_{n,n})^{\Aut(M)}$ is strict. 
\end{example}

\subsection{The rank-nullity ring of the uniform matroid}
\label{subsec:RU}
    If the matroid $M$ is the uniform matroid $U_{r,n}$, subsets of $[n]$ of size at most $r$ have all nullity zero, while subsets of size more than $r$ have all rank $r$. This means that the only generators $y_{i,j}$ of the rank-nullity ring $R^*(U_{r,n})$ that are nonzero are $y_{1,0}, y_{2,0}, \dots, y_{r,0}, y_{r,1}, y_{r,2}, \dots, y_{r,n-r}$. 
    To simplify the notation, in this case we will index the generators using just one index $l \in \{1,\dots,n\}$, as
    $$z_{l} := \sum_{S \in \binom{[n]}{l}} x_S = 
    \begin{cases} 
    y_{l,0} & \text{if } l\leq r\\
    y_{r,l-r} & \text{if } l> r
    \end{cases} \quad \in A^*(U_{n,n}).$$
    With this notation, we have
    $$R^*(U_{r,n})= \mathbb{Z}[z_1, z_2, \ldots, z_n]\subseteq A^*(U_{n,n}).$$
    Note that the rank-nullity ring $R^*(U_{r,n})$ of the uniform matroid $U_{r,n}$ does not depend on the rank $r$ but only on the size of the ground set $n$. 

\begin{prop}
\label[prop]{prop: RNU}
    The rank-nullity ring of the uniform matroid $U_{r,n}$ is equal to the $S_n$-invariant subring of the Chow ring $A^*(U_{n,n})$:
    $$R^*(U_{r,n})= A^*(U_{n,n})^{S_n}.$$
\end{prop}
\begin{proof}
    The inclusion $R^*(U_{r,n})\subseteq A^*(U_{n,n})^{S_n}$ follows from \cref{propaut}. For the reverse inclusion, 
    let $f\in A^*(U_{n,n})^{S_n}$. We can write 
\begin{equation}
\label{eq:f}
f=\sum\alpha_f(S_1, \dots, S_t, p_1, \dots, p_t) \, x_{S_1}^{p_1}x_{S_2}^{p_2}\ldots x_{S_t}^{p_t}, 
\end{equation}
where the sum is over all sequences of subsets $\emptyset \subsetneq S_1 \subsetneq \cdots \subsetneq S_t \subseteq [n]$ and all positive integers $p_1, \dots, p_k$ satisfying $p_1+\dots+p_t \leq n-1$. Since $f$ is invariant under the action of $S_n$, the coefficients $\alpha_f(S_1, \dots, S_t, p_1, \dots, p_t)$ depend only on the sequence of sizes $|S_1|, \dots, |S_t|$ and the exponents $p_1, \dots, p_t$. Letting $s_i = |S_i|$, we thus denote 
\begin{align*}
  \alpha_f(s_1, \dots, s_t, p_1, \dots, p_t) := \alpha_f(S_1, \dots, S_t, p_1, \dots, p_t).
\end{align*}
We can group together the terms in \cref{eq:f} corresponding to the same sequence of sizes $1\leq s_1< s_2<  \dots < s_t \leq n$ to obtain 
$$f = \sum_{\substack{1 \leq s_1<  \dots < s_t \leq n\\p_1+\dots+p_t \leq n-1}} 
    \left(  \alpha_f(s_1,\ldots, s_t, p_1, \ldots, p_t) \sum_{\substack{S_1\subsetneq \dots \subsetneq S_t\\
    |S_j|=s_j}} x_{S_1}^{p_1}x_{S_2}^{p_2} \dots x_{S_t}^{p_t} \right).$$
As in the proof of \cref{thmy}, we can rewrite this as
\begin{align*}
f &= \sum_{\substack{1\leq s_1 <  \dots < s_t \leq n \\p_1+\dots+p_t \leq n-1}}  \alpha_f(s_1,\ldots, s_t, p_1, \ldots p_t) \left( \sum_{
    |S_1|=s_1} \!\!x_{S_1}^{p_1} \right) \left( \sum_{
    |S_2|=s_2} \!\!x_{S_2}^{p_2} \right) \cdots \left( \sum_{|S_t|=s_t} \!\! x_{S_t}^{p_t} \right) \\ 
    &=  \sum_{\substack{1\leq s_1 <  \dots < s_t \leq n\\p_1+\dots+p_t \leq n-1}}  \alpha_f(s_1,\ldots, s_t, p_1, \ldots p_t) \, z_{s_1}^{p_1}z_{s_2}^{p_2}\cdots z_{s_t}^{p_t},
\end{align*}
which shows that $f\in R^*(U_{r,n})$.
\end{proof}

In view of the result above, the study of the rank-nullity ring $R^*(U_{r,n})=A^*(U_{n,n})^{\Aut(M)}$ of the uniform matroids is of independent interest. In the remainder of this subsection we provide a $\ZZ$-basis for it and also a Gr\"obner basis for the relations among its generators.  

\begin{prop}
\label[proposition]{prop:basisRNU}
    The set 
 $$\mathcal{B}_n=\{z_{s_1}^{p_1}z_{s_2}^{p_2}\cdots z_{s_l}^{p_l} : 1\leq p_i < s_i-s_{i-1}\ \text{and}\ 1\leq s_1<s_2<\cdots<s_l\leq n\}$$
    (where by convention $s_0 = 0$) is a $\ZZ$-basis for the rank-nullity ring $R^*(U_{r,n})$ of the uniform matroid $U_{r,n}$.
\end{prop}
\begin{proof}
       Applying \cite[Corollary 1]{feichtner2004chow} to the free matroid $U_{n,n}$, we get that the set 
    $$\mathcal{FY}_n=\{x_{S_1}^{p_1}x_{S_2}^{p_2}\cdots x_{S_l}^{p_l} : \emptyset \subsetneq S_1\subsetneq S_2 \subsetneq \cdots \subsetneq S_l \subseteq [n] \text{ and } p_i < |S_i|-|S_{i-1}|\}$$
    (where by convention $S_0 = \emptyset$) is a $\ZZ$-basis for the Chow ring $A^*(U_{n,n})$. 
    Note that the action of $S_n$ on $A^*(U_{n,n})$ permutes the elements of $\mathcal{FY}_n$.
    
    To show that the set $\mathcal{B}_n$ generates $R^*(U_{r,n})$, take $f\in R^*(U_{r,n})= A^*(U_{n,n})^{S_n}$ 
    and write it as a linear combination of the elements of $\mathcal{FY}_n$:
    \begin{align*}
        f= &\sum_{\substack{\emptyset \subsetneq S_1 \subsetneq \cdots \subsetneq S_l \subseteq [n]\\
       p_i < |S_i|-|S_{i-1}|}} \alpha_f(S_1,\dots, S_l, p_1, \dots, p_l)\,x_{S_1}^{p_1}x_{S_2}^{p_2}\dots x_{S_l}^{p_l}.
    \end{align*}
    Since $f$ is invariant under the action of $S_n$, the coefficients $\alpha_f(S_1,\dots, S_l, p_1, \dots, p_l)$ depend only on the $p_i$'s and the sizes $|S_1|, |S_2|, \dots, |S_l|$. Hence, we can group the expression above for $f$ in the same way as in the proof of \cref{prop: RNU} and obtain an expression for $f$ as a linear combination of the elements of $\mathcal{B}_n$.
    
    Finally, the set $\mathcal{B}_n$ is also linearly independent since its elements are finite sums of elements of $\mathcal{FY}_n$ without repetitions, and thus a linear dependence among the elements of $\mathcal{B}_n$ would imply a linear dependence among the elements of $\mathcal{FY}_n$. 
\end{proof}

\begin{example}
    The $\ZZ$-basis given in \cref{prop:basisRNU} for the rank-nullity ring $R^*(U_{r,5})$ of the uniform matroid $U_{r,5}$ is
        \begin{align*}
        \mathcal B_5= \{1,\,z_2,\,z_3,\,z_4,\,z_5,\,
        z_2z_4,\,z_2z_5,\,z_3^2,\,z_3z_5,\,z_4^2,\,z_5^2,\,
         z_2z_5^2,\,z_3^2z_5,\,z_4^3,\,z_5^3,\,z_5^4\}.
    \end{align*}
\end{example}

In \cref{prop:HfunctionU}, we use the basis $\mathcal B_n$ to compute simple expressions for the Hilbert function and Hilbert series of the rank-nullity ring of the uniform matroid.

We conclude this subsection with a description of all the relations among the generators $z_i$ of the rank-nullity ring of the uniform matroid. 
In fact, we give a Gr\"obner basis for the ideal of all such relations.

\begin{thm}\label{thm:grobnerbasis}
    The rank-nullity ring $R^*(U_{r,n})= A^*(U_{n,n})^{S_n}$ of the uniform matroids is isomorphic to the quotient of the free polynomial ring $\ZZ[Z_1, \dots, Z_n]$ by the ideal $\mathcal I$ generated by the polynomials
    \begin{equation}\label{eq:GB}
        Z_a \cdot \left( \sum_{i = b}^n \binom{i-a-1}{b-a-1} (Z_i + Z_{i+1} + \dots + Z_n )^{b-a} \right) \qquad \text{for $0\leq a < b \leq n$}
    \end{equation}
    (with the convention that $Z_0 = 1$), where the class $[Z_i]$ in this quotient corresponds to the element $z_i \in R^*(U_{r,n})$. 
    
    Furthermore, the polynomials in \cref{eq:GB} are a Gr\"obner basis for the ideal $\mathcal I$ with respect to any monomial order in which $Z_1 \succ Z_2 \succ \dots \succ Z_n$.
\end{thm}

\begin{proof}
We first show that the elements $z_i \in R^*(U_{r,n})$ satisfy the relations in \cref{eq:GB}. Following \cite[Theorem 1]{feichtner2004chow}, the generators of the Chow ring $A^*(U_{n,n})$ satisfy the relations
\begin{equation*}
0 = x_S \cdot \left(\sum_{T' \supseteq T} x_{T'}\right)^{|T|\setminus |S|} \qquad \text{for $\emptyset \subseteq S \subsetneq T \subseteq [n]$,}
\end{equation*}
with the convention that $x_\emptyset = 1$. Since $x_{T_1}x_{T_2} = 0$ whenever $T_1 \nsubseteq T_2$ and $T_2 \nsubseteq T_1$, using the multinomial theorem we can expand this as
\begin{equation*}
0 = x_S \cdot \sum \binom{|T|\setminus |S|}{d_1, d_2, \dots, d_k} \, x_{T_1}^{d_1} \dotsb x_{T_k}^{d_k}
\end{equation*}
where the sum is over all tuples $(T_1, \dots, T_k)$ of subsets of $[n]$ such that $T \subseteq T_1 \subsetneq T_2 \subsetneq \dots \subsetneq T_k$ and all sequences of positive integers $(d_1, \dots, d_k)$ with $d_1+ \dots + d_k = |T|\setminus |S|$. 

Now, fix $0\leq a < b \leq n$. Adding the relations above over all pairs $S \subsetneq T$ satisfying $|S|=a$ and $|T| = b$, we obtain
\begin{align*}
0 &= \sum_{\substack{S \subsetneq T\\|S|=a, \,\, |T| = b}} \left( x_S \cdot \sum_{\substack{T \subseteq T_1 \subsetneq T_2 \subsetneq \dots \subsetneq T_k\\d_1+ \dots + d_k = b-a}} \binom{b-a}{d_1, d_2, \dots, d_k} \, x_{T_1}^{d_1} \dotsb x_{T_k}^{d_k} \right).
\end{align*}
In this double summation, a term $\binom{b-a}{d_1, d_2, \dots, d_k} \, x_S \, x_{T_1}^{d_1} \dotsb x_{T_k}^{d_k}$ appears as many times as there are subsets $T$ satisfying $S \subsetneq T \subseteq T_1$, which is $\binom{|T_1|-a}{b-a}$ times. 
We can thus write the previous equation as
\begin{align*}
0 &= \sum_{\substack{S \subsetneq T_1 \subsetneq T_2 \subsetneq \dots \subsetneq T_k\\|S| = a,\,\, |T_1|\geq b,\,\, d_1+ \dots + d_k = b-a}} \binom{|T_1|-a}{b-a} \binom{b-a}{d_1, d_2, \dots, d_k} \, x_S \, x_{T_1}^{d_1} \dotsb x_{T_k}^{d_k}.
\end{align*}
Since variables corresponding to incomparable subsets have product equal to zero in the Chow ring $A^*(U_{n,n})$, we can further rewrite this as 
\begin{align*}
0 &= \sum_{\substack{b \leq t_1 < t_2 < \dots < t_k \leq n\\d_1+ \dots + d_k = b-a}} \binom{t_1-a}{b-a} \binom{b-a}{d_1, d_2, \dots, d_k} \left( \sum_{|S|=a} x_S \right) \left(\sum_{|T_1| = t_1} x_{T_1}^{d_1}\right) \dots \left(\sum_{|T_k| = t_k} x_{T_k}^{d_k}\right)\\
&= \sum_{\substack{b \leq t_1 < t_2 < \dots < t_k \leq n\\d_1+ \dots + d_k = b-a}} \binom{t_1-a}{b-a} \binom{b-a}{d_1, d_2, \dots, d_k} z_a \, z_{t_1}^{d_1} \dots z_{t_k}^{d_k}.
\end{align*}
Grouping together terms with the same value of $t_1$ and $d_1$, and then using the identity $\binom{b-a}{d_1, d_2, \dots, d_k} = \binom{b-a}{d_1} \binom{b-a-d_1}{d_2, \dots, d_k}$, we get
\begin{align*}
0 &= \sum_{\substack{t_1 \geq b\\1 \leq d_1 \leq b-a}} \left( \binom{t_1-a}{b-a} z_a \, z_{t_1}^{d_1} \cdot \sum_{\substack{t_1+1 \leq t_2 < \dots < t_k \leq n\\d_2 + \dots + d_k = b-a-d_1}} \binom{b-a}{d_1, d_2, \dots, d_k} z_{t_2}^{d_2} \dots z_{t_k}^{d_k} \right)\\
&= z_a \cdot \sum_{\substack{t_1 \geq b\\1 \leq d_1 \leq b-a}} \left( \binom{t_1-a}{b-a} \binom{b-a}{d_1} z_{t_1}^{d_1} \cdot \sum_{\substack{t_1+1 \leq t_2 < \dots < t_k \leq n\\d_2 + \dots + d_k = b-a-d_1}} \binom{b-a-d_1}{d_2, \dots, d_k} z_{t_2}^{d_2} \dots z_{t_k}^{d_k} \right)\\
&= z_a \cdot \sum_{\substack{t_1 \geq b\\1 \leq d_1 \leq b-a}} \binom{t_1-a}{b-a} \binom{b-a}{d_1} z_{t_1}^{d_1} \cdot (z_{t_1+1} + z_{t_1+2} + \dots + z_n)^{b-a-d_1}.
\end{align*}
This can be further rewritten as
\begin{align*}
0 &= z_a \sum_{t_1 \geq b} \binom{t_1-a}{b-a} \left( \sum_{d_1=1}^{b-a} \binom{b-a}{d_1} z_{t_1}^{d_1} \cdot (z_{t_1+1} + z_{t_1+2} + \dots + z_n)^{b-a-d_1} \right)\\
 &= z_a  \sum_{t_1 \geq b} \binom{t_1-a}{b-a} \left((z_{t_1} + z_{t_1+1} + z_{t_1+2} + \dots + z_n)^{b-a}- (z_{t_1+1} + z_{t_1+2} + \dots + z_n)^{b-a} \right).
\end{align*}
Finally, reordering terms yields
\begin{align*}
0 &= z_a \cdot \sum_{t_1 \geq b} \left( \binom{t_1-a}{b-a} - \binom{t_1-1-a}{b-a} \right)(z_{t_1} + z_{t_1+1} + z_{t_1+2} + \dots + z_n)^{b-a}\\
&= z_a \cdot \sum_{t_1 \geq b} \binom{t_1-1-a}{b-a-1} (z_{t_1} + z_{t_1+1} + z_{t_1+2} + \dots + z_n)^{b-a},
\end{align*}
which is the desired relation.

To show that these relations form a Gr\"obner basis, we note that the leading term of the polynomial in \cref{eq:GB} is $Z_a Z_b^{b-a}$. The ideal generated by these monomials, with $0 \leq a < b \leq n$, consists precisely of the monomials not in the basis $\mathcal B_n$ described in \cref{prop:basisRNU}, proving the statement.
\end{proof}

\begin{rem}\label[rem]{rem:klyachko}
    A slightly different presentation of the $S_n$-invariant subring $A^*(U_{n,n})^{S_n}$ of the Chow ring of the permutahedral variety appears in \cite[Section 3.1]{nadeau2023permutahedral}, following the work of Klyachko \cite{klyachko1985orbits} on permutahedral varieties of more general Coxeter types. 
    Working with $\QQ$-coefficients, the subring $A^*(U_{n,n})^{S_n}$ is equal to the polynomial ring on $n-1$ variables $\QQ[u_1, \dots, u_{n-1}]$ modulo the relations $2u_i^2= u_i u_{i-1} + u_i u_{i+1}$ (with the convention that $u_0 = u_n = 0$).
    The changes of variables linking the two presentations are 
    \begin{align*}
        u_i &= z_{n-i+1} + 2 z_{n-i+2} + 3 z_{n-i+3} + \dots + i z_n, \\
        z_i &= u_{n-i+1} - 2 u_{n-i} + u_{n-i-1}.
    \end{align*} 
\end{rem}

\begin{question}
It would be interesting to classify the matroids $M$ for which the containment $R^*(M) \subseteq A^*(U_{n,n})^{\Aut(M)}$ is an equality. As stated in \cref{prop: RNU}, this class includes all uniform matroids. 
Considering elements of degree $1$ shows that this equality implies that $M$ satisfies:
\begin{enumerate}
    \item[{\crtcrossreflabel{(*)}[item:transitivity]}] For any $0\leq i,j\leq n$, the automorphism group $\Aut(M)$ acts transitively on the collection of sets $\mathcal S_{i,j} := \{ S \subseteq [n] : \rank(S) = i \text{ and } \nulll(S) = j\}$.
\end{enumerate}
In particular, taking $i = \rank(M)$ and $j=0$, this means that $\Aut(M)$ acts transitively on the set of bases of $M$. Similarly, taking $i = \rank(M)$ and $j=1$ we see that $\Aut(M)$ acts transitively on the set of circuits of $M$, since subsets of rank $\rank(M)$ and nullity $1$ contain exactly one circuit, and all circuits are contained in such a subset.
However, it is unclear to us exactly what matroids satisfy the very restrictive property \ref{item:transitivity}. 
\end{question}

\subsection{The Hilbert function of the rank-nullity ring}
\label{subsec:H}

We now present simple expressions for the Hilbert series of the rank-nullity ring of the uniform matroid, and provide examples of the general behavior for other matroids.

\begin{prop}
\label[prop]{prop:HfunctionU}
    The Hilbert function of the rank-nullity ring $R^*(U_{r,n})= A^*(U_{n,n})^{S_n}$ of the uniform matroids is given by
    $${\rm HF}(d) := {\rm rank}_\ZZ(R^d(U_{r,n}))= \binom{n-1}{d}.$$
    Consequently, its Hilbert series is equal to 
    $${\rm HS}(t) := \sum_{d\geq0} {\rm HF}(d)\,t^d = (1+t)^{n-1}.$$
\end{prop}
\begin{proof}
By \cref{prop:basisRNU}, the following set of monomials is a $\ZZ$-basis for the degree-$d$ part of $R^*(U_{r,n})$:
$$\textstyle \mathcal{B}^d_n =\{z_{s_1}^{p_1}z_{s_2}^{p_2}\dots z_{s_l}^{p_l} : \sum p_i=d \text{ and } 1 \leq p_i < s_i-s_{i-1}\}.$$
We show that there are exactly $\binom{n-1}{d}$ monomials in $\mathcal B^d_n$ by describing a concrete bijection with the set $\binom{[n-1]}{d} = \{S\subseteq [n-1] : |S| = d\}$.
Any subset $S \in \binom{[n-1]}{d}$ naturally decomposes as a disjoint union $S = I_1 \sqcup I_2 \sqcup \dots \sqcup I_l$ of intervals of consecutive numbers; more precisely, the subsets $I_j$ have the form $I_j = \{ a_j, a_j+1, \dots, b_j\}$ and $b_j +1 < a_{j+1}$. 
The map $f:\binom{[n-1]}{d} \rightarrow \mathcal{B}^d_n$ sending the subset $S$ to the monomial $z_{s_1}^{p_1}z_{s_2}^{p_2}\cdots z_{s_l}^{p_l}$ where $s_j = b_j + 1$ and $p_j = |I_j|$ is a bijection.
Indeed, the inverse map $f^{-1}: \mathcal{B}^d_n \rightarrow \binom{[n-1]}{d}$ sends a monomial
$z_{s_1}^{p_1}z_{s_2}^{p_2}\cdots z_{s_l}^{p_l}$ to the subset $S = I_1 \sqcup I_2 \sqcup \dots \sqcup I_l$ 
where $I_j = \{s_j-p_j, s_j - p_j +1, \dots, s_j - 1\}$.  
\end{proof}

\begin{example}\label[example]{ex:hilbertfunctions}
Let $n=5$. The bijection $f$ described in the proof of \cref{prop:HfunctionU} between subsets of $\{1,2,3,4\}$ and monomials in the basis $\mathcal B_5$ is (ignoring brackets and commas in our notation for subsets):
\begin{equation*}
    \begin{array}{ccc}
    \emptyset & \mapsto & 1
    \end{array}
    \quad
    \begin{array}{ccc}
        1 & \mapsto & z_2 \\
        2 & \mapsto & z_3 \\ 
        3 & \mapsto & z_4 \\
        4 & \mapsto & z_5
    \end{array}
    \quad
    \begin{array}{ccc}
    12 &\mapsto & z_3^2 \\
    1 \,\,\, 3 &\mapsto & z_2z_4 \\
    1 \,\,\, 4 &\mapsto & z_2z_5 \\
    23 &\mapsto & z_4^2 \\
    2 \,\,\, 4 &\mapsto & z_3z_5 \\
    34 &\mapsto & z_5^2
    \end{array}
    \quad
    \begin{array}{ccc}
    123 &\mapsto & z_4^3 \\
    12 \,\,\, 4 &\mapsto & z_3^2z_5 \\ 
    1 \,\,\, 34 &\mapsto & z_2z_5^2 \\   
    234 &\mapsto & z_5^3  
    \end{array}
    \quad
    \begin{array}{ccc}
    1234 &\mapsto & z_5^4   
    \end{array}.
\end{equation*}  
\end{example}

In the following example we see that, for general matroids, the Hilbert function of the rank-nullity ring is not necessarily symmetric.  

\begin{example}
\cref{tableH} shows the Hilbert function of the rank-nullity ring $R^*(M)$ of some matroids $M$, computed using Macaulay2.
These are displayed as the vector ${\rm HF}_{R^*(M)} := ({\rm rank}_\ZZ ({R^0(M)}), \dots, {\rm rank}_\ZZ ({R^{n-1}(M)}))$.
Recall that the rank-nullity ring of the uniform matroids $R^*(U_{r,n})$ depends only on the size $n$ of the ground set, hence in this case we do not specify the rank.
Our list includes all non-isomorphic matroids of size at most 4, but only some of the matroids of higher ranks.
For reference, we also list the Hilbert function ${\rm HF}_{A^*(U_{n,n})}$ of the Chow ring $A^*(U_{n,n})$ of the free matroid on the same ground set, of which these rings are subrings. 

In the table we use the following notation: 
\begin{itemize}
    \item $M_1$ is the rank-2 matroid on 4 elements with only one non-basis. 
    \item $M_2$ is the rank-3 matroid on 5 elements with two non-bases $\{1,2,3\}, \{3,4,5\}$. This is a graphic matroid, associated to the graph $K_4$ minus an edge.
    \item $M_3$ is the rank-2 matroid on 6 elements with three non-bases $\{1,2\},\{3,4\},\{5,6\}$. This is the graphic matroid associated to a triangle graph with double edges.
    \item $M(K_4)$ is the rank-3 graphic matroid on 6 elements associated to the graph $K_4$.
    \item $M_4$ is the rank-4 matroid on 6 elements obtained by taking $U_{4,5}$ and replacing an element by 2 parallel elements. This is the graphic matroid associated to a pentagon graph where one of the sides is a double edge. 
    \item $F_7^{-}$ is the non-Fano matroid.\qedhere
\end{itemize}
\end{example}

    \begin{table}[htp]
        \begin{tabular}{c|c|c}
             \multirow{2}{*}{$M$} & \multirow{2}{*}{${\rm HF}_{R^*(M)}$} & \multirow{2}{*}{${\rm HF}_{A^*(U_{n,n})}$} \\ 
             & & \\
             \hline \hline 
             $U_{r,1}$ & $(1)$ & \multirow{1}{*}{$(1)$}\\ 
             \hline
             $U_{r,2}$ & $(1,1)$ & \multirow{2}{*}{$(1,1)$}\\
             $U_{0,1}\oplus U_{1,1}$ & $(1,1)$ & \\
             \hline
             $U_{r,3}$ & $(1,2,1)$ & \multirow{5}{*}{$(1,4,1)$}\\ 
             $U_{0,1}\oplus U_{1,2}$ & $(1,3,1)$ & \\
             $U_{0,1}\oplus U_{0,1}\oplus U_{1,1}$ & $(1,3,1)$ & \\
             $U_{0,1}\oplus U_{1,1}\oplus U_{1,1}$ & $(1,3,1)$ & \\
             $U_{1,1}\oplus U_{1,2}$ & $(1,3,1)$ & \\
             \hline
             $U_{r,4}$ & $(1,3,3,1)$ & \multirow{13}{*}{$(1,11,11,1)$}\\ 
             $U_{0,1}\oplus U_{1,3}$ & $(1,4,5,1)$ &\\ 
             $U_{0,1}\oplus U_{0,1} \oplus U_{1,2}$ & $(1,5,6,1)$ &\\ 
             $U_{0,1}\oplus U_{0,1} \oplus U_{0,1} \oplus U_{1,1}$ & $(1,5,5,1)$ &\\ 
             $M_1$ & $(1,4,5,1)$ &\\ 
             $U_{1,2}\oplus U_{1,2}$ & $(1,4,4,1)$ & \\
             $U_{0,1}\oplus U_{2,3}$ & $(1,5,5,1)$ & \\
             $U_{1,1}\oplus U_{1,3}$ & $(1,5,5,1)$ & \\
              $U_{0,1}\oplus U_{1,1}\oplus U_{1,2}$ & $(1,6,8,1)$ & \\
               $U_{0,1}\oplus U_{0,1}\oplus U_{1,1}\oplus U_{1,1}$ & $(1,6,6,1)$ & \\
            $U_{1,1}\oplus U_{2,3}$ & $(1,4,5,1)$ & \\
             $U_{1,1}\oplus U_{1,1}\oplus U_{1,2}$ & $(1,5,6,1)$ & \\
            $U_{0,1}\oplus U_{1,1}\oplus U_{1,1}\oplus U_{1,1}$ & $(1,5,5,1)$ & \\
             \hline
             $U_{r,5}$& $(1,4,6,4,1)$ & \multirow{8}{*}{$(1,26,66,26,1)$}\\ 
              $U_{0,1}\oplus U_{0,1}\oplus U_{1,3}$ & $(1,6,12,9,1)$ & \\
 $U_{0,1}\oplus U_{0,1}\oplus U_{0,1}\oplus U_{1,2}$ &$(1,7,14,9,1)$ & \\ 
  $U_{0,1}\oplus U_{0,1}\oplus U_{0,1}\oplus U_{0,1}\oplus U_{1,1}$ &$(1,7,12,7,1)$ & \\ 
              $M_2$ & $(1,5,9,7,1)$ &\\
    $U_{0,1}\oplus U_{2,4}$ &$(1, 6, 11, 7, 1)$ & \\
            $U_{0,1}\oplus U_{1,1} \oplus U_{2,3}$ & $(1,8,18,12,1)$ \\
             $U_{1,1}\oplus U_{3,4}$ & $(1,5,9,7,1)$ & \\
             \hline
             $U_{r,6} $& $(1,5,10,10,5,1)$ & \multirow{8}{*}{$(1,57,302,302,57,1)$}\\ 
             $U_{0,1}\oplus U_{0,1}\oplus U_{0,1}\oplus U_{0,1}\oplus U_{1,2}$ & $(1,9,25,28,12)$ & \\
             $M_3$ & $(1,6,14,16,9,1)$ & \\
             $M(K_4)$ & $(1,6,14,16,8,1)$ & \\
             $U_{0,1} \oplus U_{1,1} \oplus U_{2,4}$ & $(1,10,32,39,16,1)$ & \\
             $M_4$ & $(1, 8, 22, 25, 11, 1)$ & \\
             $U_{2,3}\oplus U_{2,3}$ & $(1,7,18, 20,8,1)$ & \\
             $U_{1,1}\oplus U_{4,5}$ & $(1,6,14,16,9,1)$ & \\
             \hline 
             $F_7^{-}$ & $(1,7,20,30,25,11,1)$ & \multirow{1}{*}{$(1,120,1191,2416,1191,120,1)$} \\
        \end{tabular}
        \caption{Hilbert functions of some rank-nullity rings}
        \label{tableH}
    \end{table}

As we saw in the previous examples, the top non-zero degree of the rank-nullity ring of a matroid is always equal to that of the Chow ring $A^*(U_{n,n})$.

\begin{prop}
    The rank-nullity ring $R^*(M)$ of a matroid $M$ on the ground set $[n]$ has highest non-zero degree equal to $n-1$. In other words,
    $${\rm rank}_\ZZ ({R^{n-1}(M)}) = 1.$$
\end{prop}
\begin{proof}
    Since $R^{n-1}(M) \subseteq A^{n-1}(U_{n,n}) \cong \ZZ$, it is enough to show that ${R^{n-1}(M)}\neq \{0\}$. For this, consider, for instance, the element 
    $$f := \prod_{k=1}^{n-1} \left( \sum_{i+j = k} y_{i,j} \right) = \prod_{k=1}^{n-1} \left( \sum_{|S| = k} x_S \right) \quad \in R^{n-1}(M).$$ 
    Expanding out the product on the right hand side gives an expression for $f$ as the sum of all monomials $x_{S_1}x_{S_2}\dots x_{S_{n-1}}$ where $\emptyset \subsetneq S_1 \subsetneq S_2 \subsetneq \dotsc \subsetneq S_{n-1} \subsetneq [n]$ is a maximal chain of subsets of $[n]$. Each of these monomials is equal to $1$ under the isomorphism $A^{n-1}(U_{n,n}) \cong \ZZ$ \cite[Proposition 5.8]{adiprasito2018hodge}, and thus in particular, $f$ is nonzero (in fact, it is equal to the number of maximal chains of subsets of $[n]$, which is $n!$).
\end{proof}

We note that the Hilbert functions of the rank-nullity rings displayed in \cref{tableH} are all unimodal, and even log-concave. The fact that these Hilbert functions increase up to its middle value follows from the ``K\"ahler package'' satisfied by the Chow ring $A^*(U_{n,n})$, as we now explain. 

Following \cite[Section 1.2]{adiprasito2018hodge}, we call a function $f:2^{E}\rightarrow\mathbb{R}$ \emph{strictly submodular} if 
$$f(S_1)+f(S_2)> f(S_1 \cup S_2) + f(S_1 \cap S_2)$$ 
for every pair of incomparable subsets $S_1, S_2 \subseteq E$, with the convention that $f(\emptyset)=f(E)=0$. 
An example of such a function is
\begin{equation}
\label{eq: Lef}
    f(S)=|S|\cdot |E\setminus S|.
\end{equation}
    
\begin{prop}
\label[prop]{prop:Hincresing}
    The Hilbert function of the rank-nullity ring $R^*(M)$ of a matroid $M$ on the ground set $[n]$ satisfies
    $${\rm rank}_\ZZ ({R^0(M)}) \leq {\rm rank}_\ZZ ({R^1(M)}) \leq \dots \leq {\rm rank}_\ZZ ({R^{\floor{\frac{n}{2}}}(M)}).$$
\end{prop}
\begin{proof}
  Any subset $S \subseteq [n]$ of rank $i$ and nullity $j$ must have size $i+j$. 
  This implies that the strictly submodular function $f$ in \cref{eq: Lef} is constant on subsets of the same rank and nullity. 
  As in \cite[Theorem 1.4]{adiprasito2018hodge}, the element 
  $$\ell = \sum_{\emptyset \subsetneq S \subsetneq [n]} f(S) \, x_S = \sum_{0\leq i,j \leq n} (i+j)(n-i-j) \, y_{i,j} \in R^*(M)$$ is an element of $A^1(U_{n,n})$ defining an ample class on the permutahedral toric variety, which in particular implies that, for any $q \leq (n-1)/2$, the multiplication map $L^q_{\ell} : A^q(U_{n,n}) \to A^{q+1}(U_{n,n})$ sending $a \mapsto \ell \cdot a$ is injective. Since $\ell \in R^1(M)$, this multiplication map restricts to an injection from $R^q(M)$ into $R^{q+1}(M)$, which implies the desired result.
\end{proof}

We have not pursued the following questions much further beyond the examples in \cref{tableH}.  

\begin{qn}
\label[qn]{qn:1}
Is the Hilbert function of the rank-nullity ring of a matroid $M$ always unimodal? Is it log-concave? Is it the sequence of coefficients of a real-rooted polynomial?
\end{qn}

\begin{qn}
For which matroids is the Hilbert function of its rank-nullity ring symmetric?
\end{qn}

\bibliographystyle{alpha}
\bibliography{biblio}

@article{berget2023tautological,
  title={Tautological classes of matroids},
  author={Berget, Andrew and Eur, Christopher and Spink, Hunter and Tseng, Dennis},
  journal={Inventiones mathematicae},
  volume={233},
  number={2},
  pages={951--1039},
  year={2023},
  publisher={Springer}
}

@book{oxley2006matroid,
  title={Matroid theory},
  author={Oxley, James G},
  volume={3},
  year={2006},
  publisher={Oxford University Press, USA}
}

@article{adiprasito2018hodge,
  title={Hodge theory for combinatorial geometries},
  author={Adiprasito, Karim and Huh, June and Katz, Eric},
  journal={Annals of Mathematics},
  volume={188},
  number={2},
  pages={381--452},
  year={2018},
  publisher={JSTOR}
}

@article{feichtner2004chow,
  title={Chow rings of toric varieties defined by atomic lattices},
  author={Feichtner, Eva Maria and Yuzvinsky, Sergey},
  journal={Inventiones mathematicae},
  volume={155},
  number={3},
  pages={515--536},
  year={2004},
  publisher={Springer}
}

@unpublished{caraballo2024staircase,
  title={A ``Staircase'' formula for the {C}hern-{S}chwartz-{M}ac{P}herson cycle of a matroid},
  author={{Caraballo Alba}, Franquiz and Liu, Jeffery},
  note={arXiv preprint arXiv:2409.03641},
}

@article{de2020chern,
  title={{C}hern-{S}chwartz-{M}ac{P}herson cycles of matroids},
  author={{L{\'o}pez de Medrano}, Luc{\'i}a and Rinc{\'o}n, Felipe and Shaw, Kris},
  journal={Proceedings of the London Mathematical Society},
  volume={120},
  number={1},
  pages={1--27},
  year={2020},
  publisher={Wiley Online Library}
}

@article{de1995wonderful,
  title={Wonderful models of subspace arrangements},
  author={de Concini, Corrado and Procesi, Claudio},
  journal={Selecta Mathematica},
  volume={1},
  pages={459--494},
  year={1995},
  publisher={Springer}
}

@article{ardila2006bergman,
  title={The {B}ergman complex of a matroid and phylogenetic trees},
  author={Ardila, Federico and Klivans, Caroline J.},
  journal={Journal of Combinatorial Theory, Series B},
  volume={96},
  number={1},
  pages={38--49},
  year={2006},
  publisher={Elsevier}
}

@article{backman2023simplicial,
  title={Simplicial generation of {C}how rings of matroids},
  author={Backman, Spencer and Eur, Christopher and Simpson, Connor},
  journal={Journal of the European Mathematical Society},
  volume={26},
  number={11},
  pages={4491--4535},
  year={2023}
}

@inproceedings{ardila2022geometry,
  title={The geometry of geometries: matroid theory, old and new},
  author={Ardila-Mantilla, Federico},
  booktitle={International Mathematical Union
Proc. Int. Cong. Math. 2022},
  volume={6},
  pages={4510--4541},
  year={2022},
  publisher={EMS Press}
}

@unpublished{mannino2023chern,
  title={Chern numbers of matroids},
  author={Mannino, Eline},
  note={arXiv preprint arXiv:2310.01956}  
}

@article{angarone2025chow,
  title={Chow rings of matroids as permutation representations},
  author={Angarone, Robert and Nathanson, Anastasia and Reiner, Victor},
  journal={Journal of the London Mathematical Society},
  volume={111},
  number={1},
  pages={e70039},
  year={2025},
  publisher={Wiley Online Library}
}

@article{nadeau2023permutahedral,
  title={The permutahedral variety, mixed Eulerian numbers, and principal specializations of Schubert polynomials},
  author={Nadeau, Philippe and Tewari, Vasu},
  journal={International Mathematics Research Notices},
  volume={2023},
  number={5},
  pages={3615--3670},
  year={2023},
  publisher={Oxford University Press}
}

@article{klyachko1985orbits,
  title={Orbits of a maximal torus on a flag space},
  author={Klyachko, Aleksandr Anatol’evich},
  journal={Functional analysis and its applications},
  volume={19},
  number={1},
  pages={65--66},
  year={1985},
  publisher={Springer}
}

@unpublished{stevens2021real,
  title={Real-rootedness conjectures in matroid theory},
  author={Stevens, Matthew},
  year={2021},
  note={Undergraduate thesis}
}

@article{ferroni2024valuative,
  title={Valuative invariants for large classes of matroids},
  author={Ferroni, Luis and Schr{\"o}ter, Benjamin},
  journal={Journal of the London Mathematical Society},
  volume={110},
  number={3},
  pages={e12984},
  year={2024},
  publisher={Wiley Online Library}
}

\end{document}